\documentclass[12pt,abstract]{scrartcl}
\usepackage[dvipdfmx]{graphicx}
\usepackage{natbib,amsmath,amsthm}
\usepackage{newtxtext}
\usepackage{newtxmath}

\newtheorem{lemma}{Lemma}
\newtheorem{theorem}{Theorem}

\makeatletter
\renewcommand{\theenumii}{\@roman\c@enumii}
\makeatother
\begin{document}

\title{On zero-sum game formulation of non zero-sum game\thanks{This work was supported by Japan Society for the Promotion of Science KAKENHI Grant Number 15K03481 and 18K01594.}}

\author{%
Yasuhito Tanaka\thanks{yasuhito@mail.doshisha.ac.jp}\\[.01cm]
Faculty of Economics, Doshisha University,\\
Kamigyo-ku, Kyoto, 602-8580, Japan.\\}

\date{}

\maketitle
\thispagestyle{empty}

\vspace{-.5cm}

\begin{abstract}
We consider a formulation of a non zero-sum $n$ players game by an $n+1$ players zero-sum game.  We suppose the existence of the $n+1$-th player in addition to $n$ players in the main game, and virtual subsidies to the $n$ players which is provided by the $n+1$-th player. Its strategic variable affects only the subsidies, and does not affect choice of strategies by the $n$ players in the main game. His objective function is the opposite of the sum of the payoffs of the $n$ players. We will show 1) The minimax theorem by Sion (Sion(1958)) implies the existence of Nash equilibrium in the $n$ players non zero-sum game. 2) The maximin strategy of each player in $\{1, 2, \dots, n\}$ with the minimax strategy of the $n+1$-th player is equivalent to the Nash equilibrium strategy of the $n$ players non zero-sum game. 3) The existence of Nash equilibrium in the $n$ players non zero-sum game implies Sion's minimax theorem  for pairs of each of the $n$ players and the $n+1$-th player.
\end{abstract}

\vspace{-.5cm}

\begin{description}
	\item[Keywords:] zero-sum game, non zero-sum game, minimax theorem, virtual subsidy
\end{description}

\begin{description}
	\item[JEL Classification:] C72
\end{description}

\newpage

\section{Introduction}

We consider a formulation of a non zero-sum $n$ players game by an $n+1$ players zero-sum game. We suppose the existence of the $n+1$-th player in addition to $n$ players in the main game, and virtual subsidies to the $n$ players which is provided by the $n+1$-th player. Its strategic variable affects only the subsidies, and does not affect choice of strategies by the $n$ players in the main game. His objective function is the opposite of the sum of the payoffs of the $n$ players, then the game with $n+1$ players, $n$ players in the main game and the $n+1$-th player, is a zero-sum game.

We will show the following results.

\begin{enumerate}
	\item The minimax theorem by Sion (\cite{sion}) implies the existence of Nash equilibrium in the $n$ players non zero-sum game.
	\item The maximin strategy of each player in $\{1, 2, \dots, n\}$ with the minimax strategy of the $n+1$-th player is equivalent to the Nash equilibrium strategy of the $n$ players non zero-sum game.
\item The existence of Nash equilibrium in the $n$ players non zero-sum game implies Sion's minimax theorem  for pairs of each of the $n$ players and the $n+1$-th player.
\end{enumerate}

\section{The model and the minimax theorem}

There are $n$ players Player 1, 2, $\dots$, $n$ in a non zero-sum game. The strategic variable of Player $i$ is denoted by $x_i$. The common strategy space of the players is denoted by $X$, which is a compact set. There exists another player, Player $n+1$. His strategic variable is $f$, We consider virtual subsidies to each player other than Player $n+1$, $\psi(f)$, which is provided by Player $n+1$ and is equal for any player. It is zero at the equilibrium. 

The payoff of Player $i\in \{1, 2, \dots, n\}$ is written as
\[\pi_i(x_1,x_2,\dots,x_n, f)=\varphi_i(x_1,x_2,\dots,x_n)+\psi(f),\ i\in \{1, 2, \dots, n\}.\]
The objective function of Player $n+1$ is
\[\pi_{n+1}=-(\pi_1+\pi_2+\dots \pi_n)=-\sum_{i=1}^n\varphi_i(x_1,x_2,\dots,x_n)-n\psi(f).\]
The strategy space of Player $n+1$ is denoted by $F$ which is a compact set.  Player $n+1$ is not a dummy player because he can determine the value of its strategic variable. We assume
\[\min_{f\in F}\psi(f)=0.\]
Denote 
\[a=\arg\min_{f\in F}\psi(f).\]
We postulate that this is unique. The game with Player 1, 2, $\dots$, $n$ and Player $n+1$ is a zero-sum game because
\begin{align*}
&\pi_1(x_1, x_2, \dots, x_n, f)+\pi_2(x_1, x_2, \dots, x_n, f)+\dots +\pi_{n}(x_1, x_2, \dots, x_n, f)\\
+&\pi_{n+1}(x_1, x_2, \dots, x_n, f)=0.
\end{align*}

Sion's minimax theorem (\cite{sion}, \cite{komiya}, \cite{kind}) for a continuous function is stated as follows.
\begin{lemma}
Let $X$ and $Y$ be non-void convex and compact subsets of two linear topological spaces, and let $f:X\times Y \rightarrow \mathbb{R}$ be a function that is continuous and quasi-concave in the first variable and continuous and quasi-convex in the second variable. Then
\[\max_{x\in X}\min_{y\in Y}f(x,y)=\min_{y\in Y}\max_{x\in X}f(x,y).\] \label{l1}
\end{lemma}
We follow the description of this theorem in \cite{kind}.

Let $x_k$'s for $k\neq i$ be given, then $\pi_i$ is a function of $x_i$ and $f$. We can apply Lemma \ref{l1} to such a situation, and get the following equation.
\begin{equation}
\max_{x_i\in X}\min_{f\in F}\pi_i(x_1, x_2, \dots, x_n, f)=\min_{f\in F}\max_{x_i\in X}\pi_i(x_1, x_2, \dots, x_n, f).\label{as0}
\end{equation}
We assume that $\arg\max_{x_i\in X}\min_{f\in F}\pi_i(x_1, x_2, \dots, x_n, f)$, $\arg\min_{f\in F}\max_{x_i\in X}\pi_i(x_1, x_2, \dots, x_n, f)$ and so on are unique, that is, single-valued. We also assume that the best responses of players in any situation are unique.

\section{The main results}

Choice of $f$ by Player $n+1$ has an effect only on the fixed subsidy for each player. The optimal value of $f$ for Player $n+1$, which is equal to $a$, is determined independently of $x_1$, $x_2$, $\dots$, $x_n$, and the optimal values of the strategic variables for Player 1, 2, $\dots$, $n$ are determined independently of $f$. We have
\[\pi_i(x_1,x_2,\dots, x_n,f)-\psi(f)=\pi_i(x_1,x_2,\dots, x_n,a)=\varphi_i(x_1,x_2,\dots, x_n),\ i\in \{1, 2, \dots,n\},\]
for any value of $f$. Thus,
\[\arg\max_{x_i\in X}\pi_i(x_1,x_2,\dots, x_n,f)=\arg\max_{x_i\in X}\pi_i(x_1,x_2,\dots, x_n,a)\ \mathrm{for\ any}\ f,\]
and
\begin{equation}
\arg\min_{f\in F}\pi_i(x_1,x_2,\dots, x_n,f)=a,\ i\in \{1, 2, \dots, n\}.\label{a}
\end{equation}

First we show the following result.
\begin{theorem}
\begin{enumerate}
	\item Sion's minimax theorem (Lemma \ref{l1}) implies the existence of Nash equilibrium in the non zero-sum main game.\label{one}
	\item The maximin strategy of each player in $\{1, 2, \dots, n\}$ with the minimax strategy of Player $n+1$ is equivalent to its Nash equilibrium strategy of the non zero-sum main game.\label{two}
\end{enumerate}
\label{t1}
\end{theorem}
\begin{proof}
Let $(\tilde{x}_1,\tilde{x}_2,\dots, \tilde{x}_n)$ be the solution of the following equation.
\[
\begin{cases}
\tilde{x}_1=\arg\max_{x_1\in X}\min_{f\in F}\pi_1(x_1, \tilde{x}_2,\dots, \tilde{x}_n, f)\\
\tilde{x}_2=\arg\max_{x_2\in X}\min_{f\in F}\pi_2(\tilde{x}_1, x_2,\tilde{x}_3,\dots, \tilde{x}_n, f)\\
\dots\\
\tilde{x}_n=\arg\max_{x_n\in X}\min_{f\in F}\pi_n(\tilde{x}_1, \tilde{x}_2,\dots,\tilde{x}_{n-1},x_n,f).\\
\end{cases}
\]
Then, we have
\begin{align}
&\max_{x_i\in X}\min_{f\in F}\pi_i(\tilde{x}_1, \dots, x_i,\dots, \tilde{x}_n,f)=\min_{f\in F}\pi_i(\tilde{x}_1, \dots, \tilde{x}_i,\dots, \tilde{x}_n,f)\label{t2-1}\\
=&\min_{f\in F}\max_{x_i\in X}\pi_i(\tilde{x}_1, \dots, x_i,\dots, \tilde{x}_n,f),\ i\in \{1, 2, \dots, n\}.\notag
\end{align}
Since
\[\pi_i(\tilde{x}_1, \dots, \tilde{x}_i,\dots, \tilde{x}_n,f)\leq \max_{x_i\in X}\pi_i(\tilde{x}_1, \dots, x_i,\dots, \tilde{x}_n,f),\ i\in \{1, 2, \dots, n\},\]
and
\[\min_{f\in F}\pi_i(\tilde{x}_1, \dots, \tilde{x}_i,\dots, \tilde{x}_n,f)=\min_{f\in F}\max_{x_i\in X}\pi_i(\tilde{x}_1, \dots, x_i,\dots, \tilde{x}_n,f),\ i\in \{1, 2, \dots, n\},\]
we get
\begin{align}
&\arg\min_{f\in F}\pi_i(\tilde{x}_1, \dots, \tilde{x}_i,\dots, \tilde{x}_n,f)=\arg\min_{f\in F}\max_{x_i\in X}\pi_i(\tilde{x}_1, \dots, x_i,\dots, \tilde{x}_n,f),\label{t2-2}\\
&\ i\in \{1, 2, \dots, n\}.\notag
\end{align}
Because the game is zero-sum,
\[\sum_{i=1}^n\pi_i(\tilde{x}_1, \dots, x_i,\dots, \tilde{x}_n,f)=-\pi_{n+1}(\tilde{x}_1, \dots, x_i,\dots, \tilde{x}_n,f).\]
Therefore, from (\ref{a})
\begin{align}
&\arg\min_{f\in F}\pi_i(\tilde{x}_1, \dots, x_i,\dots, \tilde{x}_n,f)\label{t2-3}\\
=&\arg\max_{f\in F}\pi_{n+1}(\tilde{x}_1, \dots, x_i,\dots, \tilde{x}_n,f)=a,\ i\in \{1, 2, \dots, n\}.\notag
\end{align}
From (\ref{t2-1}), (\ref{t2-2}) and  (\ref{t2-3}) we obtain
\begin{align}
&\min_{f\in F}\max_{x_i\in X}\pi_i(\tilde{x}_1, \dots, x_i,\dots, \tilde{x}_n,f)=\max_{x_i\in X}\pi_i(\tilde{x}_1, \dots, {x}_i,\dots, \tilde{x}_n,a)\label{t2-4}\\
=&\min_{f\in F}\pi_i(\tilde{x}_1, \dots, \tilde{x}_i,\dots, \tilde{x}_n,f)=\pi_i(\tilde{x}_1, \dots, \tilde{x}_i,\dots, \tilde{x}_n,a),\ i\in \{1, 2, \dots, n\}\notag
\end{align}
(\ref{t2-3}) and (\ref{t2-4}) mean that $(x_1, x_2,\dots,x_n, f)=(\tilde{x}_1,\tilde{x}_2,\dots,\tilde{x}_n,a)$ is a Nash equilibrium of the zero-sum game with $n+1$ players.

$\tilde{x}_1$, $\tilde{x}_2$, $\dots$, $\tilde{x}_n$ are determined independently of $f$. Thus,
\[\max_{x_i\in X}\varphi_i(\tilde{x}_1, \dots, {x}_i,\dots, \tilde{x}_n)=\varphi_i(\tilde{x}_1, \dots, \tilde{x}_i,\dots, \tilde{x}_n),\ i\in \{1, 2, \dots, n\}.\]
Therefore,  $(\tilde{x}_1,\tilde{x}_2, \dots, \tilde{x}_n)$ is a Nash equilibrium of the non zero-sum game with Player 1, 2, $\dots$, $n$.
\end{proof}

Next we show

\begin{theorem}
The existence of Nash equilibrium in the $n$ players non zero-sum game implies Sion's minimax theorem for pairs of Player $i$, $i\in \{1,2,\dots,n\}$ and Player $n+1$.
\label{t2}
\end{theorem}
\begin{proof}
Let $(\tilde{x}_1, \tilde{x}_2, \dots, \tilde{x}_n)$ be a Nash equilibrium of the $n$ players non zero-sum game. Consequently,
\begin{align*}
\varphi_i(\tilde{x}_1, \dots, \tilde{x}_i, \dots, \tilde{x}_n)\geq \varphi_i(\tilde{x}_1, \dots, x_i, \dots, \tilde{x}_n)\ \mathrm{for\ any}\ x_i,\ i\in\{1, 2, \dots, n\}.
\end{align*}
This is based on the fact that there exists a value of $x_i$, $x_i^*$, such that given $x_1$, $x_2$, $\dots$, $x_n$ other than $x_i$,
\begin{align*}
\varphi_i({x}_1, \dots, {x}^*_i, \dots, {x}_n)\geq \varphi_i({x}_1, \dots, x_i, \dots, {x}_n)\ \mathrm{for\ any}\ x_i.
\end{align*}
Thus,
\begin{align*}
\pi_i(x_1, \dots, x_i^*, \dots, x_n,f)\geq \pi_i(x_1, \dots, x_i, \dots, x_n,f)\ \mathrm{for\ any}\ x_i\ \mathrm{and\ any\ value\ of}\ f,\ i\in\{1, 2, \dots, n\},
\end{align*}
Since
\[\arg\min_{f\in F}\pi_i(x_1, \dots, x_i^*, \dots, x_n,f)=\arg\max_{f\in F}\psi(f)=a,\]
we have
\begin{align}
&\min_{f\in F}\max_{x_i\in X_i}\pi_i(x_1, \dots, x_i, \dots, x_n,f)\leq \max_{x_i\in X_i}\pi_i(x_1, \dots, x_i, \dots, x_n,a) \label{e3}\\
=&\min_{f\in F}\pi_i(x_1, \dots, x_i^*, \dots, x_n,f)\leq \max_{x_i\in X_i}\min_{x_n\in x_n}\pi_i(x_1, \dots, x_i, \dots, x_n,f),\ i\in \{1, 2, \dots, n-1\}.\notag
\end{align}
On the other hand, since
\[\min_{f\in F}\pi_i(x_1, \dots, x_i, \dots, x_n,f)\leq \pi_i(x_1, \dots, x_i, \dots, x_n,f),\]
we have
\[\max_{x_i\in X_i}\min_{f\in F}\pi_i(x_1, \dots, x_i, \dots, x_n,f)\leq \max_{x_i\in X_i}\pi_i(x_1, \dots, x_i, \dots, x_n,f).\]
This inequality holds for any $f$. Thus,
\[\max_{x_i\in X_i}\min_{f\in F}\pi_i(x_1, \dots, x_i, \dots, x_n,f)\leq \min_{f\in F}\max_{x_i\in X_i}\pi_i(x_1, \dots, x_i, \dots, x_n,f).\]
With (\ref{e3}), we obtain
\begin{equation}
\max_{x_i\in X_i}\min_{f\in F}\pi_i(x_1, \dots, x_i, \dots, x_n,f)=\min_{f\in F}\max_{x_i\in X_i}\pi_i(x_1, \dots, x_i, \dots, x_n,f),\label{t1-1}
\end{equation}
given $x_1$, $x_2$, $\dots$, $x_n$ other than $x_i$. (\ref{e3}) and (\ref{t1-1}) imply
\[\max_{x_i\in X_i}\min_{f\in F}\pi_i(x_1, \dots, x_i, \dots, x_n,f)=\max_{x_i\in X_i}\pi_i(x_1, \dots, x_i, \dots, x_n,a),\]
\[\min_{f\in F}\max_{x_i\in X_i}\pi_i(x_1, \dots, x_i, \dots, x_n,f)=\min_{f\in F}\pi_i(x_1, \dots, x^*_i, \dots, x_n,f).\]
From
\[\min_{f\in F}\pi_i(x_1, \dots, x_i, \dots, x_n,f)\leq \pi_i(x_1, \dots, x_i, \dots, x_n,a),\]
and
\[\max_{x_i\in X_i}\min_{f\in F}\pi_i(x_1, \dots, x_i, \dots, x_n,f)=\max_{x_i\in X_i}\pi_i(x_1, \dots, x_i, \dots, x_n,a),\]
we have
\[\arg\max_{x_i\in X_i}\min_{f\in F}\pi_i(x_1, \dots, x_i, \dots, x_n,f)=\arg\max_{x_i\in X_i}\pi_i(x_1, \dots, x_i, \dots, x_n,a)=x_i^*,\ i\in \{1, 2, \dots, n-1\}.\]
We also have
\[\max_{x_i\in X_i}\pi_i(x_1, \dots, x_i, \dots, x_n,f)\geq \pi_i(x_1, \dots, x_i^*, \dots, x_n,f),\]
and
\[\min_{f\in F}\max_{x_i\in X_i}\pi_i(x_1, \dots, x_i, \dots, x_n,f)=\min_{f\in F}\pi_i(x_1, \dots, x_i^*, \dots, x_n,f).\]
Therefore, we get
\[\arg\min_{f\in F}\max_{x_i\in X_i}\pi_i(x_1, \dots, x_i, \dots, x_n,f)=\arg\min_{f\in F}\pi_i(x_1, \dots, x^*_i, \dots, x_n,f)=a,\ i\in \{1, 2, \dots, n-1\}.\]
Thus, if $(x_1, x_2, \dots, x_n)=(\tilde{x}_1,\tilde{x}_2,\dots,\tilde{x}_n)$,
\[\arg\max_{x_i\in X_i}\min_{f\in F}\pi_i(\tilde{x}_1, \dots, {x}_i, \dots, \tilde{x}_n,f)=\tilde{x}_i,\ i\in \{1, 2, \dots,n\}.\]
\end{proof}

\section{An example}

Consider a three firms oligopoly with differentiated goods. There are Firm 1, 2 and 3. Assume that the inverse demand functions are
\[p_1=a-x_1-bx_2-bx_3,\]
\[p_2=a-bx_1-x_2-bx_3,\]
\[p_3=a-bx_1-bx_2-x_3,\]
with $0<b<1$. $p_1$, $p_2$, $p_3$ are the prices of the goods of Firm 1, 2, 3. $x_1$, $x_2$, $x_3$ are the outputs of the firms. The cost functions of the firms  with the subsidies are
\[c_1(x_1)=c_1x_1-(f-a)^2,\]
\[c_2(x_2)=c_2x_2-(f-a)^2,\]
and
\[c_3(x_3)=c_3x_3-(f-a)^2.\]
$f$ is a non-negative number and $a$ is a positive number. $c_1$, $c_2$, $c_3$ are constant numbers. The profits of the firms are
\begin{equation*}
\pi_1=(a-x_1-bx_2-bx_3)x_1-c_1x_1+(f-a)^2,
\end{equation*}
\begin{equation*}
\pi_2=(a-bx_1-x_2-bx_3)x_2-c_2x_2+(f-a)^2,
\end{equation*}
and
\begin{equation*}
\pi_3=(a-bx_1-bx_2-x_3)x_3-c_3x_3+(f-a)^2.
\end{equation*}
The condition for minimization of $\pi_1$ with respect to $f$ is
\[\frac{\partial \pi_1}{\partial f}=2(f-a)=0.\]
Thus, $f=a$. Substituting this into $\pi_1$,
\[\left.\pi_1\right|_{f=a}=(a-x_1-bx_2-bx_3)x_1-c_1x_1.\]
The condition for maximization of $\left.\pi_1\right|_{f=a}$ with respect to $x_1$ is
\[\frac{\partial \left.\pi_1\right|_{f=a}}{\partial x_1}=a-2x_1-bx_2-bx_3-c_1=0.\]
Thus, 
\begin{equation*}
\arg\max_{x_1\in X}\min_{f\in F}\pi_1(x_1,x_2,x_3, f)=\frac{a-c_1-bx_2-bx_3}{2}.
\end{equation*}
Similarly, we get
\begin{equation*}
\arg\max_{x_2\in X}\min_{f\in F}\pi_2(x_1,x_2,x_3, f)=\frac{a-c_2-bx_1-bx_3}{2},
\end{equation*}
\begin{equation*}
\arg\max_{x_3\in X}\min_{f\in F}\pi_3(x_1,x_2,x_3, f)=\frac{a-c_3-bx_1-bx_2}{2}.
\end{equation*}
Solving
\[x_1=\frac{a-c_1-bx_2-bx_3}{2},\]
\[x_2=\frac{a-c_2-bx_1-bx_3}{2},\]
\[x_3=\frac{a-c_3-bx_1-bx_2}{2},\]
we obtain
\[x_1=\frac{(2-b)a+bc_3+bc_2-(2+b)c_1}{2(2-b)(b+1)},\]
\[x_2=\frac{(2-b)a+bc_3+bc_1-(2+b)c_2}{2(2-b)(b+1)},\]
\[x_3=\frac{(2-b)a+bc_1+bc_2-(2+b)c_3}{2(2-b)(b+1)}.\]

They are the same as the equilibrium outputs of the oligopoly with Firm 1, 2 and 3.

In this paper we presented a zero-sum game formulation of a non zero-sum $n$ players game considering the $n+1$-th player and virtual subsidies to the players provided by the $n+1$-th player.

\end{document}